\documentclass[12pt]{article}
\usepackage[centertags]{amsmath}
\usepackage{amsfonts}
\usepackage{amssymb}
\usepackage{latexsym}
\usepackage{amsthm}
\usepackage{newlfont}
\usepackage{graphicx}
\usepackage{listings}
\usepackage{booktabs}
\usepackage{abstract}
\date{}

\newlength{\defbaselineskip}
\newcommand{\setlinespacing}[1]%
           {\setlength{\baselineskip}{#1 \defbaselineskip}}

\newcommand{\N}{{\mathbb{N}}}

\newcommand{\actaqed}{\hfill $\actabox$}
{\medskip\noindent \textit{Proof of #1. }}%
{\actaqed \medskip}

\def\cB{\mathcal B}
\def\cC{{\mathcal C}}

\def \Tr{\mathcal T}
\def \K{\mathcal K}

\def \V{\mathcal V}

\def \cF{\mathcal F}
\def \cX{\mathcal X}

\def \cM{\mathcal M}
\def\R{{\mathbb R}}
\def\Z{\mathbb Z}

\def \T{\mathbb T}

\def\bbC{\mathbb C}
\def \<{\langle}
\def\>{\rangle}

\def\ba{\mathbf a}
\def\bb{\mathbf b}
\def\bx{\mathbf x}
\def\by{\mathbf y}

\def\bk{\mathbf k}
\def\bu{\mathbf u}

\def\bm{\mathbf m}
\def\bn{\mathbf n}
\def\bs{\mathbf s}

\def\bN{\mathbf N}

\newtheorem{Theorem}{Theorem}[section]
\newtheorem{Lemma}{Lemma}[section]

\newtheorem{Remark}{Remark}[section]

\newtheorem{Corollary}{Corollary}[section]
\numberwithin{equation}{section}

\newcommand{\be}{\begin{equation}}
\newcommand{\ee}{\end{equation}}

\begin{document}

\title{Dispersion of the Fibonacci and the Frolov point sets}
\author{V.N. Temlyakov\thanks{University of South Carolina and Steklov Institute of Mathematics.  }}
\maketitle
\begin{abstract}
{It is proved that the Fibonacci and the Frolov point sets, which are known to be very good for numerical integration, have optimal rate of decay of dispersion with respect to 
the cardinality of sets. This implies that the Fibonacci and the Frolov point sets provide 
universal discretization of the uniform norm for  natural collections of subspaces of the multivariate trigonometric polynomials. It is shown how the optimal upper bounds for dispersion can be derived from the upper bounds for a new characteristic -- the smooth fixed volume discrepancy. It is proved that the Fibonacci point sets provide the universal discretization of all integral norms. }
\end{abstract}

\section{Introduction} 
\label{I} 

The concept of {\it dispersion} of a point set is an important geometric characteristic of 
a point set. It was established in a recent paper \cite{VT162} that the property of a point set to have the minimal in the sense of order dispersion is equivalent, in a certain sense, to the property of the set to provide universal discretization in the $L_\infty$ norm for  natural collections of subspaces of the multivariate trigonometric polynomials.
In this paper we study decay of dispersion of the Fibonacci and the Frolov point sets 
with respect to the cardinality of sets. We remind the definition of 
dispersion. Let $d\ge 2$ and $[0,1)^d$ be the $d$-dimensional unit cube. For $\bx,\by \in [0,1)^d$ with $\bx=(x_1,\dots,x_d)$ and $\by=(y_1,\dots,y_d)$ we write $\bx < \by$ if this inequality holds coordinate-wise. For $\bx<\by$ we write $[\bx,\by)$ for the axis-parallel box $[x_1,y_1)\times\cdots\times[x_d,y_d)$ and define
$$
\cB:= \{[\bx,\by): \bx,\by\in [0,1)^d, \bx<\by\}.
$$
For $n\ge 1$ let $T$ be a set of points in $[0,1)^d$ of cardinality $|T|=n$. The volume of the largest empty (from points of $T$) axis-parallel box, which can be inscribed in $[0,1)^d$, is called the dispersion of $T$:
$$
\text{disp}(T) := \sup_{B\in\cB: B\cap T =\emptyset} vol(B).
$$
An interesting extremal problem is to find (estimate) the minimal dispersion of point sets of fixed cardinality:
$$
\text{disp*}(n,d) := \inf_{T\subset [0,1)^d, |T|=n} \text{disp}(T).
$$
It is known that 
\be\label{1.1}
\text{disp*}(n,d) \le C^*(d)/n.
\ee
Inequality (\ref{1.1}) with $C^*(d)=2^{d-1}\prod_{i=1}^{d-1}p_i$, where $p_i$ denotes the $i$th prime number, was proved in \cite{DJ} (see also \cite{RT}). The authors of \cite{DJ} used the Halton-Hammersly set of $n$ points (see \cite{Mat}). Inequality (\ref{1.1}) with $C^*(d)=2^{7d+1}$ was proved in 
\cite{AHR}. The authors of \cite{AHR}, following G. Larcher, used the $(t,r,d)$-nets (see \cite{NX} and \cite{Mat} for results on $(t,r,d)$-nets). 

In this paper we are interested in optimal behavior of dispersion with respect to the cardinality of sets. A trivial lower bound disp*$(n,d) \ge (n+1)^{-1}$ combined with (\ref{1.1}) shows that the optimal rate of decay 
of dispersion with respect to cardinality $n$ of sets is $1/n$. In this paper we prove that 
the Fibonacci and the Frolov point sets have optimal in the sense of order rate of decay 
of dispersion. We present results on the Fibonacci point sets in Section \ref{Fib} and 
results on the Frolov point sets in Section \ref{Fro}. In Section \ref{SD} we introduce a new concept of discrepancy -- the smooth fixed volume discrepancy -- and show how 
good upper bounds of it can be used for proving optimal (in the sense of order) upper bounds for dispersion. 
These are the main results of the paper. At the end of the paper, in Section \ref{Un} we give a comment on the universal discretization of the uniform norm. In Section \ref{Lq} we prove that the Fibonacci point sets provide the universal discretization of all integral norms. The main technical result of the paper is Lemma \ref{L3.1}. This lemma is used in the direct proof of the optimal rate of convergence of dispersion of the Frolov point sets (see Theorem \ref{T4.1}). Moreover, Lemma \ref{L3.1} is used in the proof of the upper bounds for a more delicate quantity -- the smooth fixed volume discrepancy (see Theorem \ref{dT4.1}). Theorem \ref{dT4.1} implies Theorem \ref{T4.1}. We have the same phenomenon for the Fibonacci point sets: Theorem \ref{dT4.3} on the behavior of the smooth fixed volume discrepancy implies Theorem \ref{T2.1} on the behavior of dispersion. 
For further recent results on dispersion we refer the reader to papers \cite{Ull}, \cite{Rud}, \cite{Sos} and references therein. 

\section{The Fibonacci point sets}
\label{Fib}

Let $\{b_n\}_{n=0}^{\infty}$, $b_0=b_1 =1$, $b_n = b_{n-1}+b_{n-2}$,
$n\ge 2$, --
be the Fibonacci numbers. Denote the $n$th {\it Fibonacci point set} by
$$
\cF_n:= \left\{(\mu/b_n,\{\mu b_{n-1} /b_n \}),\, \mu=1,\dots,b_n\right\}.
$$
In this definition $\{a\}$ is the fractional part of the number $a$. In this section we prove 
the following upper bound on the dispersion of the $\cF_n$.

\begin{Theorem}\label{T2.1} There is an absolute constant $C$ such that for all $n$ we have
\be\label{2.1}
\text{disp}(\cF_n) \le C/b_n.
\ee
\end{Theorem}
\begin{proof} We prove bound (\ref{2.1}) for the set $\cF_{n,\pi}:= \{2\pi \bx: \bx \in \cF_n\}$. 
For the continuous functions of two
variables, which are $2\pi$-periodic in each variable, we define
cubature formulas
$$
\Phi_n(f) :=b_n^{-1}\sum_{\mu=1}^{b_n}f\bigl(2\pi\mu/b_n,
2\pi\{\mu b_{n-1} /b_n \}\bigr),
$$
 called the {\it Fibonacci cubature formulas}.   Denote 
$$\mathbf y^{\mu}:=\bigl(2\pi\mu/b_n, 2\pi\{\mu
b_{n-1}/b_n\}\bigr), \quad \mu = 1,\dots,b_n,
$$
$$
\Phi(\mathbf k) := b_n^{-1}\sum_{\mu=1}^{b_n}e^{i(\mathbf k,\mathbf y^{\mu})}.
$$
Note that
\be\label{2.2}
\Phi_n (f) =\sum_{\mathbf k}\hat f(\mathbf k)\Phi(\mathbf k),\quad \hat f(\bk) := (2\pi)^{-2}\int_{\T^2} f(\bx)e^{-i(\bk,\bx)}d\bx,
\ee
where for the sake of simplicity we may assume that $f$ is a
trigonometric polynomial. It is clear that (\ref{2.2}) holds for $f$ with absolutely convergent Fourier series. 

It is easy to see that the following relation holds
\be\label{2.3}
\Phi(\mathbf k)=
\begin{cases}
1&\quad\text{ for }\quad \mathbf k\in L(n)\\
0&\quad\text{ for }\quad \mathbf k\notin L(n),
\end{cases}
\ee
where
$$
L(n) :=\bigl\{ \mathbf k = (k_1,k_2):k_1 + b_{n-1} k_2\equiv 0
\qquad \pmod {b_n}\bigr\}.
$$
Denote $L(n)':= L(n)\backslash\{\mathbf 0\}$. For $N\in\N$ define the {\it hyperbolic cross} in dimension $d$ as follows:
$$
\Gamma(N):= \left\{\bk\in\Z^d: \prod_{j=1}^d \max(|k_j|,1) \le N\right\}.
$$
The following lemma is well known (see, for instance, \cite{TBook}).

\begin{Lemma}\label{L2.1} There exists an absolute constant $\gamma > 0$
such that for any $n > 2$ for the $2$-dimensional hyperbolic cross we have
$$
\Gamma(\gamma b_n)\cap \bigl(L(n)\backslash\{\mathbf 0\}\bigr) =
\emptyset.
$$
\end{Lemma}

For $u\in (0,1]$ define the even $2\pi$-periodic hat function $h_u(t)$, $t\in [-\pi,\pi]$, as follows: $h_u(0)=1$, $h_u(t)=0$ for $t\in [u,\pi]$, linear on $[0,u]$. Then $h_u(t) = |1-t/u|$ on $[-u,u]$ and equal to $0$ on $[-\pi,\pi]\setminus (-u,u)$. Therefore,
$$
\hat h_u(k) := (2\pi)^{-1}\int_{-\pi}^\pi h_u(t)e^{-ikt}dt = \pi^{-1}\int_0^u(1-t/u)\cos (|k|t) dt.
$$
From here, using the formula
$$
\int_0^u(1-t/u)\cos (|k|t) dt = \frac{1-\cos (|k|u)}{k^2u},\quad k\neq 0,
$$
we easily obtain the following bound for $k\neq 0$
\be\label{2.4}
|\hat h_u(k)| \le \frac{C}{|k|} \min \left(|k|u,\frac{1}{|k|u}\right).
\ee
For $\bu=(u_1,u_2)$, $\bx=(x_1,x_2)$, consider
$$
h_\bu(\bx):= h_{u_1}(x_1)h_{u_2}(x_2).
$$
We now prove that for some large enough absolute constant $c>0$ any rectangle $R$ of the form $R=(a_1,a_2)\times(v_1,v_2) \subset \T^2:=[0,2\pi]^2$ with area $|R|=c/b_n$ contains at least one point from the set $\cF_{n,\pi}$. Our proof goes by contradiction.
Let $u_1$, $u_2$ be such that 
$u_1u_2=c_0/b_n$. We choose $c_0>0$ later. Take an $R\subset [0,2\pi]^2$ and write it in the form
$R=(x^0_1-u_1,x^0_1+u_1)\times(x^0_2-u_2,x^0_2+u_2)$. 
Assuming that $R$ does not contain any points from $\cF_{n,\pi}$ we get $h_\bu(y^\mu-\bx^0)=0$ for all $\mu=1,\dots,b_n$. Then, clearly
$$
E:= (2\pi)^{-2}\int_{\T^2} h_\bu(\bx-\bx^0)d\bx - \Phi_n(h_\bu(\bx-\bx^0))
$$
\be\label{2.5}
  = (2\pi)^{-2}\int_{\T^2} h_\bu(\bx-\bx^0)d\bx = (2\pi)^{-2}u_1u_2 = \frac{c_0}{(2\pi)^2 b_n}.
\ee
To obtain a contradiction we estimate the above error $E$ of the Fibonacci cubature formula from above. 

For $\bs\in \N_0^d$ -- the set of vectors with nonnegative integer coordinates, define
$$
\rho (\bs) := \{\bk \in \Z^d : [2^{s_j-1}] \le |k_j| < 2^{s_j}, \quad j=1,\dots,d\}
$$
where $[a]$ denotes the integer part of $a$. 

By formulas (\ref{2.2}) and (\ref{2.3}) we obtain
\be\label{2.6}
E \le \sum_{\bk\neq 0} |\hat h_\bu(\bk)|\Phi(\bk) = \sum_{v=1}^\infty \sum_{\|\bs\|_1=v}\sum_{\bk\in \rho(\bs)} |\hat h_\bu(\bk)|\Phi(\bk).
\ee
Lemma \ref{L2.1} implies that if $v\neq 0$ is such that $2^v\le \gamma b_n$ then for
$\bs$ with $\|\bs\|_1=v$ we have $\rho(\bs)\subset \Gamma(\gamma b_n)$ and 
$\Phi(\bk)=0$, $\bk\in\rho(\bs)$. Let $v_0\in \N$ be the smallest number satisfying 
$2^{v_0}>\gamma b_n$. Then we have 
\be\label{2.7}
E \le \sum_{v=v_0}^\infty \sum_{\|\bs\|_1=v}\sum_{\bk\in \rho(\bs)} |\hat h_\bu(\bk)|\Phi(\bk).
\ee
Lemma \ref{L2.1} implies that for $v\ge v_0$ we have
\be\label{2.8}
|\rho(\bs)\cap L(n)| \le C_12^{v-v_0}, \quad \|\bs\|_1=v.
\ee
Relations (\ref{2.7}), (\ref{2.8}), and (\ref{2.4}) imply
\be\label{2.9}
E \le C_2 2^{-v_0}\sum_{v=v_0}^\infty \sum_{\|\bs\|_1=v}   \min\left(2^{s_1}u_1,\frac{1}{2^{s_1}u_1}\right)\min\left(2^{s_2}u_2,\frac{1}{2^{s_2}u_2}\right).
\ee

We now need the following technical lemma.
\begin{Lemma}\label{L2.2} Let $u_1$, $u_2$ be such that 
$u_1u_2\ge 2^{-v}$, $v\in \N$. Then we have
$$
\sigma(v) := \sum_{\|\bs\|_1=v}   \min\left(2^{s_1}u_1,\frac{1}{2^{s_1}u_1}\right)\min\left(2^{s_2}u_2,\frac{1}{2^{s_2}u_2}\right) \le C_3 \frac{\log(2^{v+1} u_1u_2)}{2^v u_1u_2}.
$$
\end{Lemma}
\begin{proof} Our condition $u_1u_2\ge 2^{-v}$ guarantees that we have $2^vu_1u_2 \ge 1$. 
We split $\sigma(v)=\sigma_1(v)+\sigma_2(v)+\sigma_3(v)$ into three sums with summation over $s_1$ from the following three sets
$$
S_1:= \{s_1\ge 0: 2^{s_1} \le 1/u_1\},
$$
$$
S_2:=\{s_1: 1/u_1 <2^{s_1}\le 2^vu_2\},
$$
$$
S_3:=\{s_1\le v: 2^{s_1}> 2^vu_2\}.
$$
We now estimate separately the $\sigma_i(v)$, $i=1,2,3$. Using the inequality $2^vu_1u_2 \ge 1$ mentioned above, we see  that for $s_1\in S_1$ we have
$2^{s_2}u_2\ge 1$ and therefore
\be\label{2.10}
\sigma_1(v) = \sum_{s_1\in S_1} 2^{s_1}u_1\frac{1}{2^{v-s_1}u_2} = 2^{-v}\frac{u_1}{u_2}\sum_{s_1\in S_1} 2^{2s_1}\le \frac{4}{3}(2^vu_1u_2)^{-1}.
\ee
In the same way, replacing the role of $s_1$, $u_1$ by $s_2$ and $u_2$ we obtain
\be\label{2.11}
\sigma_3(v) \le \frac{4}{3}(2^vu_1u_2)^{-1}.
\ee
Finally, for $\sigma_2(v)$ we have
$$
\sigma_2(v) = \sum_{s_1\in S_2} \frac{1}{2^{s_1}u_1}\frac{1}{2^{v-s_1}u_2} = 
(2^vu_1u_2)^{-1}|S_2| 
$$
\be\label{2.12}
\le (2^vu_1u_2)^{-1}(1+\log(2^vu_1u_2)).
\ee
Combining inequalities (\ref{2.10})--(\ref{2.12}) we complete the proof of Lemma \ref{L2.2}.

\end{proof}
We now complete the proof of Theorem \ref{T2.1}. Assume that $u_1u_2=c_0/b_n$, $c_0\ge 2$.
Then the relation $2^{v_0}\asymp b_n$, relation (\ref{2.9}) and Lemma \ref{L2.2} imply
$$
E \le C_4\frac{\log c_0}{c_0 b_n}.
$$
Obviously, this contradicts (\ref{2.5}) for large enough $c_0$. 

Theorem \ref{T2.1} is now proved.

\end{proof}

\section{Technical lemmas}
\label{Tech}

In Section \ref{Fib} we discussed the two-dimensional case. In the next Sections \ref{Fro} and \ref{SD} we discuss the general $d$-dimensional case. There we need a generalization of the two-dimensional Lemma \ref{L2.2}. This section deals with 
such a generalization. It is somewhat technically involved. We begin with some notations, which are used here. For $\bu=(u_1,\dots,u_d)\in \R^d_+$ we denote 
$\bu^d:=(u_1,\dots,u_{d-1})\in \R^{d-1}_+$. It is convenient for us to use the following notation for products
$$
pr(\bu,\nu):= u_1\cdots u_\nu,\quad \nu\le d.
$$
Thus, for instance, $pr(\bu,d)=pr(\bu,d-1)u_d$. 

We are interested in the behavior of special sums
$$
\sigma(v,\bu):= \sum_{\|\bs\|_1=v}\prod_{j=1}^d \min\left(2^{s_j}u_j,\frac{1}{2^{s_j}u_j}\right),\quad v\in\N_0.
$$
Clearly, for $d\ge 3$ we have
\be\label{3.1}
\sigma(v,\bu) = \sum_{s_d=0}^v \min\left(2^{s_d}u_d,\frac{1}{2^{s_d}u_d}\right)\sigma(v-s_d,\bu^d).
\ee

The main result of this section is the following lemma.

\begin{Lemma}\label{L3.1} Let $v\in \N_0$ and $\bu\in \R^d_+$. Then we have 
the following inequalities.

(I) Under condition $2^vpr(\bu,d)\ge 1$ we have
\be\label{3.2}
\sigma(v,\bu) \le C(d) \frac{\left(\log(2^{v+1}pr(\bu,d))\right)^{d-1}}{2^vpr(\bu,d)}.
\ee

(II) Under condition $2^vpr(\bu,d)\le 1$ we have
\be\label{3.3}
\sigma(v,\bu) \le C(d) 2^vpr(\bu,d) \left(\log\frac{2}{2^{v}pr(\bu,d)}\right)^{d-1}.
\ee

\end{Lemma}
\begin{proof} Our proof goes by induction on $d$. First, we establish Lemma \ref{L3.1} 
for $d=2$. Inequality (\ref{3.2}) follows directly from Lemma \ref{L2.2}. We now prove
inequality (\ref{3.3}). As in the proof of Lemma \ref{L2.2} we split the sum $\sigma(v,\bu)$ into three sums respectively over the index sets
$$
S'_1:= \{s_1\ge 0: 2^{s_1} \le 2^vu_2\},
$$
$$
S'_2:=\{s_1: 2^vu_2 <2^{s_1} < 1/u_1\}.
$$
Note that our 
condition $2^vpr(\bu,2)=2^vu_1u_2\le 1$ implies $2^vu_2 \le 1/u_1$. 
$$
S'_3:=\{s_1\le v: 2^{s_1}\ge 1/u_1\}.
$$
We now estimate the corresponding $\sigma'_i$, $i=1,2,3$ separately.  For the first sum we have
$$
\sigma'_1 = \sum_{s_1\in S'_1} 2^{s_1}u_1 \frac{1}{2^{v-s_1}u_2} = 2^{-v}\frac{u_1}{u_2}\sum_{s_1\in S'_1}2^{2s_1} 
$$
$$
\le \frac{4}{3}2^{-v}\frac{u_1}{u_2}(2^vu_2)^2 = \frac{4}{3}2^vu_1u_2.
$$
For the second sum we have
$$
\sigma'_2 = \sum_{s_1\in S'_2} 2^{s_1}u_1 2^{v-s_1}u_2 = 2^vu_1u_2|S'_2| \le 2^vu_1u_2 \left(1+\log\frac{1}{2^vu_1u_2}\right).
$$
The third sum $S'_3$ is similar to the sum $S'_1$. In $S'_1$ we have a condition, that can be rewritten as $2^{s_2}u_2\ge 1$ and in $S'_3$ we have the condition $2^{s_1}u_1\ge 1$. Thus, for the third sum we have 
$$
\sigma'_3 \le \frac{4}{3}2^vu_1u_2.
$$
Summing up the above bounds for $\sigma'_i$, $i=1,2,3,$ we obtain (\ref{3.3}) in case 
$d=2$, which completes the proof of Lemma \ref{L3.1} in case $d=2$. 

We now proceed to the induction step from $d-1$ to $d$. Suppose Lemma \ref{L3.1} holds for 
$d-1\ge 2$. We derive from here Lemma \ref{L3.1} for $d$. We begin with the case (I), i.e. assume that inequality $2^vpr(\bu,d)\ge 1$ holds. We use identity (\ref{3.1}) and 
Lemma \ref{L3.1} for $d-1$. We split the sum $\sigma(v,\bu)$ into three sums over 
the following index sets
$$
U_1:= \{s_d\ge 0: 2^{s_d}\le 1/u_d\},
$$
$$
U_2:= \{s_d: 1/u_d < 2^{s_d} < 2^vpr(\bu,d-1)\}.
$$
Note that our assumption $2^vpr(\bu,d)\ge 1$ guarantees $1/u_d \le 2^vpr(\bu,d-1)$.
$$
U_3:=\{s_d\le v: 2^{s_d} \ge 2^vpr(\bu,d-1)\}.
$$
Then, for the first sum we have
$$
\sigma_1(v,\bu)= \sum_{s_d\in U_1} 2^{s_d}u_d\sigma(v-s_d,\bu^d).
$$
For $s_d\in U_1$ we have $2^{s_d}u_d \le 1$, which combined with the condition (I) inequality 
$$
2^{v-s_d}pr(\bu,d-1)2^{s_d}u_d = 2^vpr(\bu,d) \ge 1
$$
implies $2^{v-s_d}pr(\bu,d-1) \ge 1$. Thus, applying inequality (\ref{3.2}) of Lemma \ref{L3.1} for $d-1$ we get (for convenience, here and later we write $\alpha \ll \beta$ 
instead of $\alpha \le C(d)\beta$)
$$
\sigma_1(v,\bu) \ll \sum_{s_d\in U_1} 2^{s_d}u_d \frac{(\log(2^{v+1-s_d}pr(\bu,d-1)))^{d-2}}{2^{v-s_d}pr(\bu,d-1)} 
$$
\be\label{3.5'}
= \frac{u_d}{2^vpr(\bu,d-1)}\sum_{s_d\in U_1} 2^{2s_d} (\log(2^{v+1-s_d}pr(\bu,d-1)))^{d-2}.
\ee
Here we need the following simple technical lemma, which we formulate without proof.

\begin{Lemma}\label{L3.3} Let two numbers $A\ge 1$ and $B\ge A$ be given. Then for $\nu\in\N$  we have
$$
\sum_{k:2^k\le A} 2^{2k}\left(\log\frac{2B}{2^{k}}\right)^\nu \le C(\nu) A^{2}\left(\log\frac{2B}{A}\right)^\nu.
$$
\end{Lemma}
Using Lemma \ref{L3.3} we continue relation (\ref{3.5'})
$$
\ll \frac{u_d}{2^vpr(\bu,d-1)}\left(\frac{1}{u_d}\right)^2\log(2^{v+1}pr(\bu,d))^{d-2}
$$
$$
= 
\frac{1}{2^vpr(\bu,d)} \log(2^{v+1}pr(\bu,d))^{d-2}.
$$

Next, for the second sum we have
$$
\sigma_2(v,\bu) = \sum_{s_d\in U_2} \frac{1}{2^{s_d}u_d}\sigma(v-s_d,\bu^d).
$$
For $s_d\in U_2$ we have $2^{s_d} < 2^vpr(\bu,d-1)$, which implies $2^{v-s_d}pr(\bu,d-1)>1$. Therefore, we continue, using the first inequality of Lemma \ref{L3.1} for $d-1$.
$$
\sigma_2(v,\bu) \ll \sum_{s_d\in U_2} \frac{1}{2^{s_d}u_d}\frac{(\log(2^{v+1-s_d}pr(\bu,d-1))^{d-2}}{2^{v-s_d}pr(\bu,d-1)}
$$
$$
= \frac{1}{2^vpr(\bu,d)} \sum_{s_d\in U_2}  (\log(2^{v+1-s_d}pr(\bu,d-1))^{d-2}
$$
$$
\le \frac{1}{2^vpr(\bu,d)}  (\log(2^{v+1}pr(\bu,d))^{d-2}|U_2| \le \frac{\log(2^{v+1}pr(\bu,d))^{d-1}}{2^vpr(\bu,d)}.
$$

Finally, for the third sum we have
$$
\sigma_3(v,\bu) = \sum_{s_d\in U_3} \frac{1}{2^{s_d}u_d}\sigma(v-s_d,\bu^d).
$$
For $s_d\in U_3$ we have $2^{s_d} \ge 2^vpr(\bu,d-1)$, which is the same as \newline $2^{v-s_d}pr(\bu,d-1) \le 1$. Applying inequality (\ref{3.3}) of Lemma \ref{L3.1} for $d-1$ we obtain
$$
\sigma_3(v,\bu) \ll \sum_{s_d\in U_3} \frac{1}{2^{s_d}u_d}2^{v-s_d}pr(\bu,d-1) \left(\log\frac{2}{2^{v-s_d}pr(\bu,d-1)}\right)^{d-2}
$$
\be\label{3.4}
=\frac{2^vpr(\bu,d-1)}{u_d}\sum_{s_d\in U_3} 2^{-2s_d}\left(\log\frac{2^{s_d+1}}{2^vpr(\bu,d-1)}\right)^{d-2}.
\ee
Here we need the following simple technical lemma, which we formulate without proof.

\begin{Lemma}\label{L3.2} Let two numbers $A\ge 1$ and $0<B\le A$ be given. Then for $\nu\in \N$ we have
$$
\sum_{k:2^k\ge A} 2^{-2k}\left(\log\frac{2^{k+1}}{B}\right)^\nu \le C(\nu) A^{-2}\left(\log\frac{2A}{B}\right)^\nu.
$$
\end{Lemma}

Using Lemma \ref{L3.2} we continue relation (\ref{3.4})
$$
\ll \frac{2^vpr(\bu,d-1)}{u_d}\frac{1}{(2^vpr(\bu,d-1))^2} = \frac{1}{2^vpr(\bu,d)}.
$$ 

Combining the above inequalities for all three sums $\sigma_i(v,\bu)$ we complete the 
proof of Lemma \ref{L3.1} in the first case (I). 

We now proceed to the second case (II). In this case we split the summation over $s_d$ into three index sets:
$$
U'_1:= \{s_d\ge 0: 2^{s_d}\le 2^vpr(\bu,d-1)\},
$$
$$
U'_2:=\{s_d: 2^vpr(\bu,d-1)< 2^{s_d}<1/u_d\},
$$
$$
U'_3:=\{s_d\le v: 2^{s_d}\ge 1/u_d\}.
$$
We now estimate the corresponding sums separately. For $s_d\in U'_1$ we have 
$2^{s_d}u_d \le 2^vpr(\bu,d) \le 1$ by the condition for case (II). Also, for $s_d\in U'_1$ we have $2^{v-s_d}pr(\bu,d-1)\ge 1$. Therefore, by (\ref{3.2}) of Lemma \ref{L3.1} for $d-1$ we get
$$
\sigma'_1(v,\bu) \ll \sum_{s_d\in U'_1}2^{s_d}u_d\frac{\left(\log(2^{v+1-s_d}pr(\bu,d-1))\right)^{d-2}}{2^{v-s_d}pr(\bu,d-1)}
$$
\be\label{3.5}
= \frac{u_d}{2^vpr(\bu,d-1)}\sum_{s_d\in U'_1} 2^{2s_d}\left(\log(2^{v+1-s_d}pr(\bu,d-1))\right)^{d-2}.
\ee
 
Using Lemma \ref{L3.3} we continue relation (\ref{3.5})
$$
\ll \frac{u_d}{2^vpr(\bu,d-1)} (2^vpr(\bu,d-1))^2 = 2^vpr(\bu,d).
$$

For $s_d\in U'_2$ we have $2^{v-s_d}pr(\bu,d-1)<1$. Therefore, by (\ref{3.3}) of Lemma \ref{L3.1} for $d-1$ we obtain
$$
\sigma'_2(v,\bu) \ll \sum_{s_d\in U'_2} 2^{s_d}u_d2^{v-s_d}pr(\bu,d-1)\left(\log\frac{2}{2^{v-s_d}pr(\bu,d-1)}\right)^{d-2}
$$
$$
=2^vpr(\bu,d)\sum_{s_d\in U'_2} \left(\log\frac{2}{2^{v-s_d}pr(\bu,d-1)}\right)^{d-2}
$$
$$
\le 2^vpr(\bu,d)\left(\log\frac{2}{2^{v}pr(\bu,d)}\right)^{d-2}|U'_2| \le 2^vpr(\bu,d)\left(\log\frac{2}{2^{v}pr(\bu,d)}\right)^{d-1}.
$$

Finally, for $s_d\in U'_3$ the case (II) assumption $2^vpr(\bu,d)\le 1$ implies $2^{v-s_d}pr(\bu,d-1)\le 1$. Thus, using inequality (\ref{3.3}) from Lemma \ref{L3.1} for $d-1$ 
we get
$$
\sigma'_3(v,\bu) \ll \sum_{s_d\in U'_3}\frac{1}{2^{s_d}u_d}2^{v-s_d}pr(\bu,d-1)\left(\log\frac{2}{2^{v-s_d}pr(\bu,d-1)}\right)^{d-2}
$$
$$
= \frac{2^{v}pr(\bu,d-1)}{u_d}\sum_{s_d\in U'_3} 2^{-2s_d}\left(\log\frac{2}{2^{v-s_d}pr(\bu,d-1)}\right)^{d-2}.
$$
Using Lemma \ref{L3.2} we continue
$$
\ll  \frac{2^{v}pr(\bu,d-1)}{u_d} \left(\frac{1}{u_d}\right)^{-2}\left(\log\frac{2}{2^{v}pr(\bu,d)}\right)^{d-2} 
$$
$$
= 2^{v}pr(\bu,d) \left(\log\frac{2}{2^{v}pr(\bu,d)}\right)^{d-2}.
$$

Combining the above inequalities for all three sums $\sigma'_i(v,\bu)$ we complete the 
proof of Lemma \ref{L3.1} in the second case (II) and complete the proof of Lemma \ref{L3.1}. 

\end{proof}

\begin{Remark}\label{R3.1} It is easy to check that the above proof of Lemma \ref{L3.1} allows us to obtain the following bound on the constant $C(d) \le C^d$ 
with an absolute constant $C$.
\end{Remark}

\section{The Frolov point sets}
\label{Fro}

In this section we study dispersion of point sets, which are known to be very good 
for numerical integration, -- the {\it Frolov point sets}. We refer the reader for detailed 
presentation of the theory of the Frolov cubature formulas to \cite{TBook}, \cite{T11}, \cite{MUll}, and \cite{DTU}. We begin with a description of the Frolov point sets. The following lemma plays a fundamental role in the
construction of such point sets (see \cite{TBook} for its proof).

\begin{Lemma}\label{L4.1} There exists a matrix $A$ such that the lattice
$L(\mathbf m) = A\mathbf m$
$$
 L(\mathbf m) =
\begin{bmatrix}
L_1(\mathbf m)\\
\vdots\\
L_d(\mathbf m)
\end{bmatrix},
$$
where $\mathbf m$ is a (column)
vector with integer coordinates,
has the following properties

{$1^0$}. $\qquad \left |\prod_{j=1}^d L_j(\mathbf m)\right|\ge 1$
for all $\mathbf m \ne \mathbf 0$;

{$2^0$} each parallelepiped $P$ with volume $|P|$
whose edges are parallel
to the coordinate axes contains no more than $|P| + 1$ lattice
points.
\end{Lemma}

Let $a > 1$ and $A$ be the matrix from Lemma \ref{L4.1}. We consider the
cubature formula
$$
\Phi(a,A)(f) := \bigl(a^d |\det A|\bigr)^{-1}\sum_{\mathbf m\in\Z^d}f
\left (\frac{(A^{-1})^T\mathbf m}{a}\right)
$$
for $f$ with compact support.   

We call the {\it Frolov point set} the following set associated with the matrix $A$ and parameter $a$
$$
\cF(a,A) := \left\{\left (\frac{(A^{-1})^T\mathbf m}{a}\right)\right\}_{\bm\in\Z^d}\cap [0,1]^d =: \{z^\mu\}_{\mu=1}^N.
$$
 Clearly, the number $N=|\cF(a,A)|$ of points of this
set does not exceed $C(A)a^d$. 

The main result of this section is the following theorem.

\begin{Theorem}\label{T4.1} Let $A$ be a matrix from Lemma \ref{L4.1}. There is a  constant $C(d,A)$, which may only depend on $A$ and $d$, such that for all $a$ we have
\be\label{4.1}
\text{disp}(\cF(a,A)) \le C(A,d)a^{-d}.
\ee
\end{Theorem}
\begin{proof} The idea of the proof of this theorem is the same as of the proof of Theorem \ref{T2.1}. 
For $\bu=(u_1,\dots,u_d)\in\R^d_+$, $\bx=(x_1,\dots,x_d)$, consider
\be\label{4.1'}
h_\bu(\bx):=\prod_{j=1}^d h_{u_j}(x_j),
\ee
where $h_u(t) = |1-t/u|$ on $[-u,u]$ and equal to $0$ for $|t|\ge u$.
We now prove that for some large enough constant $C(A,d)>0$ any $d$-dimensional box $B$ of the form $B=\prod_{j=1}^d (a_1^j,a_2^j) \subset [0,1]^d$ with area $|B|=C(A,d)a^{-d}$ contains at least one point from the set $\cF(a,A)$. Our proof goes by contradiction.
Let $\bu$ be such that 
$pr(\bu,d)=c_0a^{-d}$. We choose $c_0>0$ later. Take a $B\subset [0,1]^d$ and write it in the form
$B=\prod_{j=1}^d(x^0_j-u_j,x^0_j+u_j)$. 
Assuming that $B$ does not contain any points from $\cF(a,A)$ we get $h_\bu(z^\mu-\bx^0)=0$ for all $\mu=1,\dots,N$. Then, clearly
$$
e:= \int_{[0,1]^d} h_\bu(\bx-\bx^0)d\bx - \Phi(a,A)(h_\bu(\bx-\bx^0))
$$
\be\label{4.2}
  =  \int_{[0,1]^d} h_\bu(\bx-\bx^0)d\bx = pr(\bu,d) = c_0 a^{-d}.
\ee
To obtain a contradiction we estimate the above error $e$ of the Frolov cubature formula from above. Denote for $f\in L_1 (\R^d)$
$$
\hat f(\mathbf y) := \int_{\R^d} f(\mathbf x)e^{-2\pi i(\mathbf y,\mathbf x)}d\mathbf x.
$$
For a function $f$ with finite support and absolutely convergent 
series \newline $\sum_{\mathbf m\in\Z^d}\hat f(aA\mathbf m)$ we have for the error of the Frolov  cubature formula (see \cite{TBook})
\be\label{4.3}
 \Phi(a,A)(f) -\hat f(\mathbf 0) =\sum_{\mathbf m\ne\mathbf 0}
\hat f(aA\mathbf m).
\ee
The proof of this formula is based on the Poisson formula, which we formulate in the
form convenient for us (see \cite{TBook} for the proof).  
\begin{Lemma}\label{L4.2} Let $f(\mathbf x)$ be continuous and have
compact support and the series $\sum_{\mathbf k\in\Z^d}\hat f(\mathbf k)$
converges. Then
$$
\sum_{\mathbf k\in\Z^d}\hat f(\mathbf k) =
\sum_{\bn\in\Z^d} f(\bn) .
$$
\end{Lemma}

By (\ref{4.3}) we obtain
\be\label{4.4}
e \le \sum_{\bm\neq 0} |\hat h_\bu(aA\bm)| = \sum_{v=1}^\infty \sum_{\|\bs\|_1=v}\sum_{\bm: aA\bm \in \rho(\bs)} |\hat h_\bu(aA\bm)|.
\ee
Lemma \ref{L4.1} implies that if $v\neq 0$ is such that $2^v< a^d$ then for
$\bs$ with $\|\bs\|_1=v$ there is no $\bm$ such that  $aA\bm\in\rho(\bs)$. Let $v_0\in \N$ be the smallest number satisfying 
$2^{v_0}\ge a^d$. Then we have 
\be\label{4.5}
e \le \sum_{v=v_0}^\infty \sum_{\|\bs\|_1=v}\sum_{\bm: aA\bm\in \rho(\bs)} |\hat h_\bu(aA\bm)|.
\ee
Lemma \ref{L4.1} implies that for $v\ge v_0$ we have
\be\label{4.6}
|\rho(\bs)\cap \{aA\bm\}_{\bm\in \Z^d}| \le C_12^{v-v_0}, \quad \|\bs\|_1=v.
\ee
Relations (\ref{4.5}), (\ref{4.6}), and (\ref{2.4}) imply
\be\label{4.7}
e \le C_2 2^{-v_0}\sum_{v=v_0}^\infty \sum_{\|\bs\|_1=v}   \prod_{j=1}^d\min\left(2^{s_j}u_j,\frac{1}{2^{s_j}u_j}\right).
\ee
We now assume that $c_0\ge 2$ and $c_0a^{-d}\ge 2^{-v_0}$. 
Using inequality (\ref{3.2}) of Lemma \ref{L3.1} and an analog of Lemma \ref{L3.2} we obtain from here
$$
e \le C(d,A)2^{-v_0} \frac{(\log c_0)^{d-1}}{c_0},
$$
which is in contradiction with (\ref{4.2}) for large enough $c_0$.

\end{proof}

\begin{Remark}\label{R4.1} The above proof of Theorem \ref{T4.1} and Remark \ref{R3.1} allow us to obtain the following bound on the constant $C(A,d) \le d^{c(A)d}$
with $c(A)$ depending only on $A$. 
\end{Remark}
\begin{Remark}\label{R4.2} Right after the first version of this paper, which contained 
Theorem \ref{T4.1}, has been published in arXiv Mario Ullrich informed me that he 
has an unpublished note, where he obtained a bound similar to (\ref{4.1}). His argument 
is based on different ideas. It certainly does not apply to the study of the smooth fixed volume discrepancy (see Section \ref{SD} below). 
\end{Remark}

\section{A remark on smooth discrepancy}
\label{SD}

We begin with a classical definition of discrepancy ("star discrepancy", $L_\infty$-discrepancy) of a point set $T:=\Xi_m := \{\xi^\mu\}_{\mu=1}^m\subset [0,1)^d$. 
Introduce a class of special $d$-variate characteristic functions
$$
\chi^d := \{\chi_{[\mathbf 0,\bb)}(\bx):=\prod_{j=1}^d \chi_{[0,b_j)}(x_j),\quad b_j\in [0,1),\quad j=1,\dots,d\}
$$
where $\chi_{[a,b)}(x)$ is a univariate characteristic function of the interval $[a,b)$. 
The classical definition of discrepancy of a set $T$ of points $\{\xi^1,\dots,\xi^m\}\subset [0,1)^d$ is as follows
$$
D(T,m,d)_\infty := \max_{\bb\in [0,1)^d}\left|\prod_{j=1}^db_j -\frac{1}{m}\sum_{\mu=1}^m \chi_{[\mathbf 0,\bb)}(\xi^\mu)\right|.
$$
It is equivalent within multiplicative constants, which may only depend on $d$, to the following definition
\be\label{d4.1}
D^1(T):=  \sup_{B\in\cB}\left|vol(B)-\frac{1}{m}\sum_{\mu=1}^m \chi_B(\xi^\mu)\right|,
\ee
where for $B=[\ba,\bb)\in \cB$ we denote $\chi_B(\bx):= \prod_{j=1}^d \chi_{[a_j,b_j)}(x_j)$. We use here definition (\ref{d4.1}) because it is more in a spirit of the definition of dispersion. Moreover, we consider the following optimized version of $D^1(T)$
\be\label{d4.1'}
D^{1,o}(T):= \inf_{w_1,\dots,w_m} \sup_{B\in\cB}\left|vol(B)- \sum_{\mu=1}^m w_\mu \chi_B(\xi^\mu)\right|.
\ee

We now modify definitions (\ref{d4.1}) and (\ref{d4.1'}), replacing the characteristic function $\chi_B$ by a smoother hat function $h_B$. Let $B\in\cB$ be written in the form
$$
B=\prod_{j=1}^d [x^0_j -u_j,x^0_j+u_j).
$$
Then we define
$$
h_B(\bx) := pr(\bu,d)h_\bu(\bx-\bx^0),
$$
where $h_\bu(\bx)$ is defined in (\ref{4.1'}).

The $2$-smooth discrepancy is now defined as
\be\label{d4.2}
D^2(T):=  \sup_{B\in\cB}\left|\int h_B(\bx)d\bx-\frac{1}{m}\sum_{\mu=1}^m h_B(\xi^\mu)\right|
\ee
and its optimized version as 
\be\label{d4.2'}
D^{2,o}(T):=  \inf_{w_1,\dots,w_m}\sup_{B\in\cB}\left|\int h_B(\bx)d\bx- \sum_{\mu=1}^m w_\mu h_B(\xi^\mu)\right|.
\ee
Note that the known concept of $r$-discrepancy with $r=2$ (see, for instance, \cite{TBook} and \cite{T11}) is close to the above concepts of $2$-smooth discrepancy. 

Along with $D^2(T)$ and $D^{2,o}(T)$ we consider a more refined quantity -- {\it $2$-smooth fixed volume discrepancy} -- defined as follows
\be\label{d4.3}
D^2(T,V):=  \sup_{B\in\cB:vol(B)=V}\left|\int h_B(\bx)d\bx-\frac{1}{m}\sum_{\mu=1}^m h_B(\xi^\mu)\right|;
\ee
\be\label{d4.3'}
D^{2,o}(T,V):=  \inf_{w_1,\dots,w_m}\sup_{B\in\cB:vol(B)=V}\left|\int h_B(\bx)d\bx- \sum_{\mu=1}^m w_\mu h_B(\xi^\mu)\right|.
\ee
Clearly,
$$
D^2(T) = \sup_{V\in(0,1]} D^2(T,V).
$$

{\bf The Frolov point sets.}
The main result of this section is the following Theorem \ref{dT4.1} on the Frolov point set $T=\cF(a,A)$.

\begin{Theorem}\label{dT4.1} There exists a constant $c(d,A)>0$ such that for any $V\ge V_0:= c(d,A)a^{-d}$ we have
\be\label{d5.7}
D^{2,o}(\cF(a,A),V) \le C(d,A)a^{-2d} (\log(2V/V_0))^{d-1}.
\ee
\end{Theorem}
In the definition (\ref{d4.3'}) of the quantity $D^{2,o}(T,V)$ we optimize over weights 
$w_1,\dots,w_m$, when $V$ is fixed. Therefore, the optimal weights may depend on
parameter $V$. We prove a somewhat stronger version of Theorem \ref{dT4.1}
where the weights, which provide the bound  (\ref{d5.7}), do not depend on $V$. 
We formulate it as a theorem.

\begin{Theorem}\label{dT4.2} There exists a constant $c(d,A)>0$ such that for any $V\ge V_0:= c(d,A)a^{-d}$ we have for all $B\in\cB$, $vol(B)=V$,
\be\label{d5.8}
|\Phi(a,A)(h_B) - \hat h_B(\mathbf 0)| \le C(d,A)a^{-2d} (\log(2V/V_0))^{d-1}.
\ee
\end{Theorem}

\begin{proof} By (\ref{4.5}) we have for the error (with $M:=a^d|\det A|$)
$$
\delta_B := \left|\int h_B(\bx)d\bx-\frac{1}{M}\sum_{\mu=1}^N h_B(z^\mu)\right| \le \sum_{v=v_0}^\infty \sum_{\|\bs\|_1=v}\sum_{\bm: aA\bm\in \rho(\bs)} |\hat h_B(aA\bm)|.
$$
Using (\ref{4.6}) we obtain by (\ref{2.4})
$$
\delta_B \ll \sum_{v=v_0}^\infty \sum_{\|\bs\|_1=v} 2^{v-v_0}pr(\bu,d)2^{-v}\prod_{j=1}^d\min\left(2^{s_j}u_j,\frac{1}{2^{s_j}u_j}\right).
$$
We now assume that the constant $c(d,A)$ is such that  $V_0= 2^{d-v_0}$. Then for $B\in\cB$ such that $vol(B)\ge V_0$ we have $pr(\bu,d)\ge 2^{-v_0}$. 
Using inequality (\ref{3.2}) of Lemma \ref{L3.1} and an analog of Lemma \ref{L3.2} we obtain from here
$$
 \delta_B \ll 2^{-v_0}\sum_{v=v_0}^\infty 2^{-v}\left(\log\left(2^{v+1}pr(\bu,d)\right)\right)^{d-1} 
 $$
 $$
 \ll 2^{-2v_0}\left(\log\left(2V/V_0\right)\right)^{d-1} \le a^{-2d}\left(\log\left(2V/V_0\right)\right)^{d-1}.
$$
 
\end{proof}

We now make comments on the relation between discrepancy and dispersion. It is obvious from (\ref{d4.1}) that
\be\label{d5.9}
\text{disp}(T) \le D^1(T).
\ee
The best known upper bounds for discrepancy $D^1(T)$ for sets of cardinality $m$ are of the form $D^1(T) \ll m^{-1}(\log m)^{d-1}$. Also, the classical result of Roth \cite{Ro} gives the lower bound $D^1(T) \gg m^{-1}(\log m)^{(d-1)/2}$. There are very interesting improvements of the above lower bound (see \cite{Sch1}, \cite{BL}, \cite{BLV}), which we 
do not discuss here. Therefore, inequality (\ref{d5.9}) can give us the bound 
disp$(T) \ll m^{-1}(\log m)^{d-1}$ and, for sure, we cannot get the bound disp$(T) \ll m^{-1}$ on this way.

Relations 
\be\label{d5.10}
vol(B) = 2^dpr(\bu,d) \quad \text{and}\quad \int h_B(\bx)d\bx = pr(\bu,d)^2
\ee
imply that
\be\label{d5.11}
\text{disp}(T) \le 2^d (D^2(T))^{1/2}.
\ee
This inequality is better than (\ref{d5.9}) but still cannot give us the desired bound disp$(T) \ll m^{-1}$. Thus, the step from discrepancy to smooth discrepancy does not solve the problem. It turns out that the critical step here is to the smooth fixed volume discrepancy. It is clear that 
\be\label{d5.12}
\text{disp}(T)=:V \le 2^d (D^{2,o}(T,V))^{1/2}.
\ee
Inequality (\ref{d5.12}) applied to $T=\cF(a,A)$ combined with Theorem \ref{dT4.1} 
gives for $V:=\text{disp}(T)$
$$
 \text{either}\quad V\le V_0\quad \text{or}\quad V \ll (\log(2V/V_0))^{(d-1)/2} V_0,
$$
which implies
$$
\text{disp}(\cF(a,A)) \ll V_0 \ll a^{-d} \asymp |\cF(a,A)|^{-1}.
$$

{\bf The Fibonacci point sets.}
We have discussed above new concepts of discrepancy and their applications for the
upper bounds for dispersion, in particular, for the Frolov point sets. The crucial role in 
the proof of Theorem \ref{dT4.2} is played by Lemma \ref{L3.1}. In the same way, using Lemma \ref{L2.2} instead of Lemma \ref{L3.1} we can prove the following version of Theorem \ref{dT4.2} for the Fibonacci point sets. 

\begin{Theorem}\label{dT4.3} Let $d=2$. There exists an absolute constant $c>0$ such that for any $V\ge V_0:= c/b_n$ we have for all $B\in\cB$, $vol(B)=V$
\be\label{d5.13}
\left|b_n^{-1}\sum_{\mu=1}^{b_n} h_B(\mu/b_n,\{\mu b_{n-1}/b_n\}) - \hat h_B(\mathbf 0)\right| \le C\log(2V/V_0)/b_n^2.
\ee
\end{Theorem}

Theorem \ref{dT4.3} implies Theorem \ref{T2.1}. Also, Theorem \ref{dT4.3} 
provides the following inequalities for the Fibonacci point sets $\cF_n$
$$
D^{2,o}(\cF_n,V) \le D^{2}(\cF_n,V) \le C(\log(2V/V_0))/b_n^2,\qquad V\ge V_0.
$$

\section{Generalization for higher smoothness}
\label{rSD}

In the definition of $D^1(T)$ and $D^{1,o}(T)$ -- the $1$-smooth discrepancy -- we used as a building block the univariate characteristic function $\chi_{[-u/2,u/2)}(x)$ identified by the parameter $u\in \R_+$. In numerical integration $L_1$-smoothness of a function plays an important role. A characteristic function of an interval has smoothness $1$ in the $L_1$ norm. This is why we call the corresponding discrepancy characteristics the $1$-smooth discrepancy. In the definition of $D^2(T)$,
$D^{2,o}(T)$, $D^2(T,V)$, and $D^{2,o}(T,V)$ we use the hat function 
$h_{[-u,u)}(x) =|u-x|$ for $|x|\le u$ and $h_{[-u,u)}(x) =0$ for $|x|\ge u$ instead of the characteristic function $\chi_{[-u/2,u/2)}(x)$. Function $h_{[-u,u)}(x)$ has smoothness $2$ in $L_1$. This fact gives the corresponding name. Note that
$$
h_{[-u,u)}(x) = \chi_{[-u/2,u/2)}(x) \ast \chi_{[-u/2,u/2)}(x),
$$
where
$$
f(x)\ast g(x) := \int_\R f(x-y)g(y)dy.
$$
Now, for $r=1,2,3,\dots$ we inductively define
$$
h^1(x,u):= \chi_{[-u/2,u/2)}(x),\qquad h^2(x,u):= h_{[-u,u)}(x),
$$
$$
h^r(x,u) := h^{r-1}(x,u)\ast h^1(x,u),\qquad r=3,4,\dots.
$$
Then $h^r(x,u)$ has smoothness $r$ in $L_1$ and has support $(-ru/2,ru/2)$. 
Represent a box $B\in\cB$ in the form
$$
B= \prod_{j=1}^d [x^0_j-ru_j/2,x^0_j+ru/2)
$$
and define
$$
h^r_B(\bx):= \prod_{j=1}^d h^r(x_j-x^0_j,u_j).
$$
We define the quantities $D^r(T)$,
$D^{r,o}(T)$, $D^r(T,V)$, and $D^{r,o}(T,V)$ replacing $h_B(\bx)$ by $h^r_B(\bx)$ in the definitions (\ref{d4.2}) -- (\ref{d4.3'}). By the properties of convolution we obtain
$$
\hat h^r(y,u)  = \hat h^{r-1}(y,u) \hat h^1(y,u),
$$
which implies for $y\neq 0$
$$
\hat h^r(y,u) = \left(\frac{\sin(\pi yu)}{\pi y}\right)^r.
$$
Therefore,
$$
|\hat h^r(y,u)| \le \min\left(|u|^r,\frac{1}{|y|^r}\right) = \left(\frac{|u|}{|y|}\right)^{r/2}\min\left(|yu|^{r/2},\frac{1}{|yu|^{r/2}}\right).
$$
Consider
$$
\sigma^r(v,\bu):= \sum_{\|\bs\|_1=v}\prod_{j=1}^d \min\left((2^{s_j}u_j)^{r/2},\frac{1}{(2^{s_j}u_j)^{r/2}}\right),\quad v\in\N_0.
$$
In the case $r=2$ we have $\sigma^2(v,\bu) = \sigma(v,\bu)$ with $\sigma(v,\bu)$ defined and estimated in Section \ref{Tech}. In the same way as Lemma \ref{L3.1} has been proved we can prove the following its generalization for all $r\in \N$. 

\begin{Lemma}\label{dL6.1} Let $v\in \N_0$ and $\bu\in \R^d_+$. Then we have 
the following inequalities.

(I) Under condition $2^vpr(\bu,d)\ge 1$ we have
\be\label{d3.2}
\sigma^r(v,\bu) \le C(d) \frac{\left(\log(2^{v+1}pr(\bu,d))\right)^{d-1}}{(2^vpr(\bu,d))^{r/2}}.
\ee

(II) Under condition $2^vpr(\bu,d)\le 1$ we have
\be\label{d3.3}
\sigma^r(v,\bu) \le C(d) (2^vpr(\bu,d))^{r/2} \left(\log\frac{2}{2^{v}pr(\bu,d)}\right)^{d-1}.
\ee

\end{Lemma}

Using Lemma \ref{dL6.1} instead of Lemma \ref{L3.1} we prove the following generalization of Theorem \ref{dT4.2}. 

\begin{Theorem}\label{dT6.1} Let $r\ge 2$. There exists a constant $c(d,A,r)>0$ such that for any $V\ge V_0:= c(d,A,r)a^{-d}$ we have for all $B\in\cB$, $vol(B)=V$,
\be\label{d6.1}
|\Phi(a,A)(h^r_B) - \hat h^r_B(\mathbf 0)| \le C(d,A,r)a^{-rd} (\log(2V/V_0))^{d-1}.
\ee
\end{Theorem}

\begin{Corollary}\label{dC6.1} For $r\ge2$ there exists a constant $c(d,A,r)>0$ such that for any $V\ge V_0:= c(d,A,r)a^{-d}$ we have
\be\label{d6.2}
D^{r,o}(\cF(a,A),V) \le C(d,A,r)a^{-rd} (\log(2V/V_0))^{d-1}.
\ee
\end{Corollary}
 
Similar generalizations can be obtained for the Fibonacci point sets. Using Lemma \ref{dL6.1} instead of Lemma \ref{L2.2} we obtain the following results. 

\begin{Theorem}\label{dT6.2} Let $d=2$, $r\ge2$. There exists a   constant $c(r)>0$ such that for any $V\ge V_0:= c(r)/b_n$ we have for all $B\in\cB$, $vol(B)=V$
\be\label{d6.3}
\left|b_n^{-1}\sum_{\mu=1}^{b_n} h^r_B(\mu/b_n,\{\mu b_{n-1}/b_n\}) - \hat h^r_B(\mathbf 0)\right| \le C(r)\log(2V/V_0)/b_n^r.
\ee
\end{Theorem}

  Theorem \ref{dT6.2} 
provides the following inequalities for the Fibonacci point sets $\cF_n$ in case $r\ge 2$
$$
D^{r,o}(\cF_n,V) \le D^{r}(\cF_n,V) \le C(r)(\log(2V/V_0))/b_n^r,\qquad V\ge V_0.
$$

\section{Universal discretization of the uniform norm}
\label{Un}

In this section we demonstrate an application of results from Sections \ref{Fib} and \ref{Fro} to the problem of universal discretization. For a more detailed discussion of 
universality in approximation and learning theory we refer the reader to \cite{Tem16}, 
\cite{TBook}, \cite{T11}, \cite{DTU}, \cite{VT162}, \cite{GKKW},
\cite{BCDDT}, \cite{VT113}. We remind the discretization problem setting, which we plan to discuss (see \cite{VT160} and \cite{VT161}). 

{\bf Marcinkiewicz problem.} Let $\Omega$ be a compact subset of $\R^d$ with the probability measure $\mu$. We say that a linear subspace $X_N$ of the $L_q(\Omega)$, $1\le q < \infty$, admits the Marcinkiewicz-type discretization theorem with parameters $m$ and $q$ if there exist a set $\{\xi^\nu \in \Omega, \nu=1,\dots,m\}$ and two positive constants $C_j(d,q)$, $j=1,2$, such that for any $f\in X_N$ we have
\be\label{d1.1}
C_1(d,q)\|f\|_q^q \le \frac{1}{m} \sum_{\nu=1}^m |f(\xi^\nu)|^q \le C_2(d,q)\|f\|_q^q.
\ee
In the case $q=\infty$ we define $L_\infty$ as the space of continuous on $\Omega$ functions and ask for 
\be\label{d1.2}
C_1(d)\|f\|_\infty \le \max_{1\le\nu\le m} |f(\xi^\nu)| \le  \|f\|_\infty.
\ee
We will also use a brief way to express the above property: the $\cM(m,q)$ theorem holds for  a subspace $X_N$ or $X_N \in \cM(m,q)$. 

{\bf Universal discretization problem.} This problem is about finding (proving existence) of 
a set of points, which is good in the sense of the above Marcinkiewicz-type discretization 
for a collection of linear subspaces (see \cite{VT162}). We formulate it in an explicit form. Let $\cX_N:= \{X_N^j\}_{j=1}^k$ be a collection of linear subspaces $X_N^j$ of the $L_q(\Omega)$, $1\le q \le \infty$. We say that a set $\{\xi^\nu \in \Omega, \nu=1,\dots,m\}$ provides {\it universal discretization} for the collection $\cX_N$ if, in the case $1\le q<\infty$, there are two positive constants $C_i(d,q)$, $i=1,2$, such that for each $j\in [1,k]$ and any $f\in X_N^j$ we have
\be\label{1.1u}
C_1(d,q)\|f\|_q^q \le \frac{1}{m} \sum_{\nu=1}^m |f(\xi^\nu)|^q \le C_2(d,q)\|f\|_q^q.
\ee
In the case $q=\infty$  for each $j\in [1,k]$ and any $f\in X_N^j$ we have
\be\label{1.2u}
C_1(d)\|f\|_\infty \le \max_{1\le\nu\le m} |f(\xi^\nu)| \le  \|f\|_\infty.
\ee

In \cite{VT162} we studied the universal discretization  for the collection of subspaces of trigonometric polynomials with frequencies from parallelepipeds (rectangles). For $\bs\in\N^d_0$
define
$$
R(\bs) := \{\bk \in \Z^d :   |k_j| < 2^{s_j}, \quad j=1,\dots,d\}.
$$
    Let $Q$ be a finite subset of $\Z^d$. We denote
$$
\Tr(Q):= \{f: f=\sum_{\bk\in Q}c_\bk e^{i(\bk,\bx)}\}.
$$
Consider the collection $\cC(n,d):= \{\Tr(R(\bs)), \|\bs\|_1=n\}$.

The following theorem was proved in \cite{VT162}.
\begin{Theorem}\label{T5.1} Let a set $T$ with cardinality $|T|= 2^r=:m$ have dispersion 
satisfying the bound disp$(T) < C(d)2^{-r}$ with some constant $C(d)$. Then there exists 
a constant $c(d)\in \N$ such that the set $2\pi T:=\{2\pi\bx: \bx\in T\}$ provides the universal discretization in $L_\infty$ for the collection $\cC(n,d)$ with $n=r-c(d)$.
\end{Theorem}

Theorem \ref{T5.1} in a combination with Theorems \ref{T2.1} and \ref{T4.1} guarantees that the appropriately chosen Fibonacci ($d=2$) and  Frolov (any $d\ge 2$) point sets provide 
universal discretization in $L_\infty$ for the collection $\cC(n,d)$.

\section{Universal discretization of the $L_q$ norm}
\label{Lq}

We begin this section with proving a general conditional result. Then we derive from 
it universality of the Fibonacci point sets for discretization of the $L_q$ norm for all $1\le q\le \infty$. We formulate the universality problem with weights.

{\bf Marcinkiewicz problem with weights.}   We say that a linear subspace $X_N$ of the $L_q(\Omega)$, $1\le q < \infty$, admits the weighted Marcinkiewicz-type discretization theorem with parameters $m$ and $q$ if there exist a set of knots $\{\xi^\nu \in \Omega\}$, a set of weights $\{w_\nu\}$, $\nu=1,\dots,m$, and two positive constants $C_j(d,q)$, $j=1,2$, such that for any $f\in X_N$ we have
\be\label{6.1}
C_1(d,q)\|f\|_q^q \le  \sum_{\nu=1}^m w_\nu |f(\xi^\nu)|^q \le C_2(d,q)\|f\|_q^q.
\ee
Then we also say that the $\cM^w(m,q)$ theorem holds for  a subspace $X_N$ or $X_N \in \cM^w(m,q)$. 
Obviously, $X_N\in \cM(m,q)$ implies that $X_N\in \cM^w(m,q)$. 

{\bf Universal discretization problem with weights.} This problem is about finding (proving existence) of 
a set of points and a set of weights  which are good in the sense of the above Marcinkiewicz-type discretization with weights
for a collection of linear subspaces. We formulate it in an explicit form. Let $\cX_N:= \{X_N^j\}_{j=1}^k$ be a collection of linear subspaces $X_N^j$ of the $L_q(\Omega)$, $1\le q < \infty$. We say that a set of knots $\{\xi^\nu \in \Omega\}$ and a set of weights $\{w_\nu\}$, $\nu=1,\dots,m$, provide {\it universal discretization with weights} for the collection $\cX_N$ if  there are two positive constants $C_i(d,q)$, $i=1,2$, such that for each $j\in [1,k]$ and any $f\in X_N^j$ we have
\be\label{6.2}
C_1(d,q)\|f\|_q^q \le   \sum_{\nu=1}^m w_\nu|f(\xi^\nu)|^q \le C_2(d,q)\|f\|_q^q.
\ee
 
 For a set of knots $\Xi_m:=\{\xi^\nu\}_{\nu=1}^m\subset  \Omega$ and a set of weights $W_m:=\{w_\nu\}_{\nu=1}^m$ consider the cubature formula
 \be\label{6.3}
 I_m:=I_m(\Xi_m,W_m)(f) := \sum_{\nu=1}^m w_\nu f(\xi^\nu).
 \ee
 For $\bN\in \N^d_0$ define a subspace of trigonometric polynomials
 $$
 \Tr(\bN) := \left\{f(\bx): f(\bx)= \sum_{\bk: |k_j|\le N_j,j=1,\dots,d} c_\bk e^{i(\bk,\bx)}\right\}.
 $$
The following Lemma \ref{L6.1} is the conditional result that we mentioned above.

\begin{Lemma}\label{L6.1} Let $\bN\in \N^d_0$ and let a set of knots $\Xi_m:=\{\xi^\nu\}_{\nu=1}^m\subset \T^d$ and a set of nonnegative weights $W_m:=\{w_\nu\}_{\nu=1}^m$
be such that for any $f\in \Tr(3\bN)$ we have
\be\label{6.4}
I_m(\Xi_m,W_m)(f) = (2\pi)^{-d}\int_{\T^d} f(\bx)d\bx.
\ee
Then for any $1\le q\le \infty$ we have for all $f\in \Tr(\bN)$
\be\label{6.5}
C_1(d)\|f\|_q \le   \left(\sum_{\nu=1}^m w_\nu|f(\xi^\nu)|^q\right)^{1/q} \le C_2(d)\|f\|_q
\ee
with constants $C_1(d)$ and $C_2(d)$, which may only depend on $d$.
\end{Lemma}
\begin{proof}  We need some classical trigonometric polynomials. We begin with the univariate case. 
 The Dirichlet kernel of order $n$:
$$
\mathcal D_n (x):= \sum_{|k|\le n}e^{ikx} = e^{-inx} (e^{i(2n+1)x} - 1)
(e^{ix} - 1)^{-1} 
$$
\be\label{a1.1}
=\bigl(\sin (n + 1/2)x\bigr)\bigm/\sin (x/2)
\ee
   is an even trigonometric polynomial.  The Fej\'er kernel of order $n - 1$:
$$
\mathcal K_{n} (x) := n^{-1}\sum_{k=0}^{n-1} \mathcal D_k (x) =
\sum_{|k|\le n} \bigl(1 - |k|/n\bigr) e^{ikx} 
$$
$$
=\bigl(\sin (nx/2)\bigr)^2\bigm /\bigl(n (\sin (x/2)\bigr)^2\bigr).
$$
The Fej\'er kernel is an even nonnegative trigonometric
polynomial in $\Tr(n-1)$.  It satisfies the obvious relations
\be\label{FKm}
\| \mathcal K_{n} \|_1 = 1, \qquad \| \mathcal K_{n} \|_{\infty} = n.
\ee
The de la Vall\'ee Poussin kernel
$$
\mathcal V_{n} (x) := n^{-1}\sum_{l=n}^{2n-1} \mathcal D_l (x)= 2\K_{2n}(x)-\K_n(x) 
$$
is an even trigonometric
polynomial of order $2n - 1$.

In the multivariate case  define the Fej\'er and de la Vall\'ee Poussin kernels as follows:
$$
\K_\bN(\bx):=\prod_{j=1}^d \K_{N_j}(x_j), \qquad \mathcal V_{\mathbf N} (\bx) := \prod_{j=1}^d \mathcal V_{N_j}  (x_j)  ,\qquad
\mathbf N = (N_1 ,\dots,N_d) .
$$
For $f\in\Tr(\bN)$ we have for each $\bx\in\T^d$ that $f(\by)\V_\bN(\bx-\by)\in \Tr(3\bN)$ and by our condition (\ref{6.4}) we obtain
\be\label{6.6}
f(\bx) = (2\pi)^{-d}\int_{\T^d} f(\by)\V_\bN(\bx-\by)d\by = \sum_{\nu=1}^m w_\nu f(\xi^\nu)\V_\bN(\bx-\xi^\nu).
\ee
Define a space 
$$
\ell_q(W_m):=\{\ba\in \bbC^m\} \quad\text{with norm}\quad   \|\ba\|_{q,w}:=\left(\sum_{\nu=1}^m w_\nu |a_\nu|^q\right)^{1/q} .
$$
Let $V_\bN$ be the operator on $\ell_q(W_m)$ defined as follows:
$$
V_\bN(\ba) := \sum_{\nu=1}^{m} w_\nu a_\nu \mathcal V_\bN (\bx - \xi^\nu).
$$
Property (\ref{FKm}) implies that $\|\V_\bN\|_1 \le 3^d$. Therefore,
\be\label{6.7}
\| V_\bN \|_{\ell_1(W_m)\to L_1}\le 3^d.
\ee
We now bound the norm $\| V_\bN \|_{\ell_\infty(W_m)\to L_\infty}$. Clearly,
\be\label{6.8}
\| V_\bN \|_{\ell_\infty(W_m)\to L_\infty}\le \left\|\sum_{\nu=1}^{m} w_\nu | \mathcal V_\bN (\bx - \xi^\nu)|\right\|_\infty.
\ee
We need the following simple technical lemma.

\begin{Lemma}\label{L6.2} Under conditions of Lemma \ref{L6.1} we have
$$
\left\|\sum_{\nu=1}^{m} w_\nu | \mathcal V_\bN (\bx - \xi^\nu)|\right\|_\infty \le 3^d.
$$
\end{Lemma}
\begin{proof} Represent
$$
\V_{N_j}(t) = 2\K_{2N_j}(t)-\K_{N_j}(t).
$$
Using the fact that the Fej{\'e}r kernel is a nonnegative polynomial we obtain 
$$
\sum_{\nu=1}^{m} w_\nu | \mathcal V_\bN (\bx - \xi^\nu)| \le \sum_{\nu=1}^{m} w_\nu  \prod_{j=1}^d \left(2\K_{2N_j}(x_j - \xi^\nu_j) + \K_{N_j}(x_j - \xi^\nu_j)\right)
$$
and by (\ref{6.4}) and (\ref{FKm}) we continue
$$
\le \sum_{k=0}^d \binom{d}{k} 2^k 1^{d-k} =3^d.
$$
\end{proof}

Lemma \ref{L6.2} and inequality (\ref{6.7}) imply by the Riesz-Thorin interpolation theorem that 
\be\label{6.9}
\| V_\bN \|_{\ell_q(W_m)\to L_q}\le 3^d,\qquad 1\le q\le \infty.
\ee
By representation (\ref{6.6}) and inequality (\ref{6.9}) we obtain for $f\in \Tr(\bN)$
$$
\|f(\bx)\|_q = \left\|\sum_{\nu=1}^m w_\nu f(\xi^\nu)\V_\bN(\bx-\xi^\nu)\right\|_q \le 3^d \left(\sum_{\nu=1}^m w_\nu |f(\xi^\nu)|^q\right)^{1/q},
$$
which proves the first inequality in (\ref{6.5}) of Lemma \ref{L6.1}. 

We now prove the second inequality in (\ref{6.5}) of Lemma \ref{L6.1} for $1\le q < \infty$. In the case $q=\infty$ it is trivial. We have ($q':=\frac{q}{q-1}$)
\begin{align*}
\sum_{\nu=1}^m w_\nu |f(\xi^\nu)|^q & =
\sum_{\nu=1}^{m}w_\nu f(\xi^\nu)\varepsilon_\nu \bigl| f(\xi^\nu) \bigr|^{q-1} =\\
&=(2\pi)^{-d}\int_{\T^d}f(\bx)\sum_{\nu=1}^{m}w_\nu\varepsilon_\nu
\bigl|f(\xi^\nu)\bigr|^{q-1}\mathcal V_\bN(\bx-\xi^\nu)d\bx\le\\
&\le \|f \|_q\left \|\sum_{\nu=1}^{m}w_\nu\varepsilon_\nu \bigl| f(\xi^\nu)
\bigr|^{q-1}\mathcal V_\bN (\bx - \xi^\nu) \right\|_{q'}.
\end{align*}
Using (\ref{6.9}) we see that the last term is
$$
\le 3^d \| f \|_q \left(\sum_{\nu=1}^m w_\nu \bigl| f(\xi^\nu)
\bigr|^{q}\right)^{(q-1)/q} ,
$$
which implies the required inequality.

Lemma \ref{L6.1} is proved.

\end{proof}

{\bf Universality of the Fibonacci point sets.} We use Lemmas \ref{L2.1} and \ref{L6.1}.
Lemma \ref{L2.1} and identity (\ref{2.2}) imply that for any 
$$
f\in \Tr(\Gamma(\gamma b_n)):= \left\{f: f(\bx)= \sum_{\bk\in \Gamma(\gamma b_n)} c_\bk e^{(\bk,\bx)}\right\}
$$
we have
$$
\Phi_n(f) = \hat f(\mathbf 0).
$$
Therefore, condition (\ref{6.4}) of Lemma \ref{L6.1} is satisfied for $m=b_n$,
$I_m=\Phi_n$, $w_\nu =1/m$, $\nu=1,\dots,m$, with $\bN=(2^{s_1},\dots,2^{s_d})$
under condition $\bs\in \N^d_0$ is such that $3\cdot 2^{\|\bs\|_1} \le \gamma b_n$. 
Lemma \ref{L6.1} implies the following result.

\begin{Theorem}\label{T7.1} The Fibonacci point set $\cF_n$ provides the universal discretization in $L_q$, $1\le q\le\infty$, for the collection $\cC(r,2)$ with $r$ satisfying 
the condition $3\cdot 2^r \le \gamma b_n$.
\end{Theorem}

{\bf Acknowledgment.} The author would like to thank the Erwin Schr{\"o}dinger International Institute
for Mathematics and Physics (ESI) at the
University of Vienna for support. This paper was completed, when the author participated in 
the ESI-Semester "Tractability of High Dimensional Problems and Discrepancy",
September 11--October 13, 2017. 


\end{document}